\documentclass[english]{amsart}
\usepackage{amsmath}
\usepackage{amssymb}
\usepackage{amsthm}
\usepackage{amscd}
\usepackage{babel}
\usepackage[latin1]{inputenc}
\usepackage{graphicx}

\newtheorem{teor}{Theorem}

\newtheorem{cor}{Corollary}
\newtheorem{prop}{Proposition}
\newtheorem{con}{Conjecture}
\newtheorem{lem}{Lemma}

\theoremstyle{definition}

\newtheorem{q}{Question}
\newtheorem*{rem}{Remark}

\renewcommand{\subjclassname}{AMS \textup{2010} Mathematics Subject
Classification\ }

\author{Jos\'{e} Mar\'{i}a Grau}
\address{Departamento de Matemáticas, Universidad de Oviedo\\ Avda. Calvo Sotelo s/n, 33007 Oviedo, Spain}
\email{grau@uniovi.es}

\author{Antonio M. Oller-Marc\'{e}n}
\address{Centro Universitario de la Defensa de Zaragoza\\ Ctra. Huesca s/n, 50090 Zaragoza, Spain} \email{oller@unizar.es}

\title{Power sums over finite commutative unital rings}

\begin{document}

\begin{abstract}
In this paper we compute the sum of the $k$-th powers over any finite commutative unital rings, thus generalizing known results for finite fields, the rings of integers modulo $n$ or the ring of Gaussian integers modulo $n$. As an application we focus on quotient rings of the form $\mathbb{Z}/n\mathbb{Z}[x]/(f(x))$ for any polynomial $f\in\mathbb{Z}[x]$.
\end{abstract}

\maketitle \subjclassname{13B99, 13A99, 13F99}


\section{Introduction}

For a finite ring $R$ and $k\geq 1$, we define the power sum
$$S_k (R):= \sum_{r\in R} r^k.$$
Throughout the paper we will deal only with finite commutative unital rings and our main objective will be the computation of $S_k (R)$ in such case.

The problem of computing $S_k (R)$ is completely solved only for some particular families of finite rings. If $R$ is a finite field $\mathbb{F}_q$, the value of $S_k (\mathbb{F}_q)$ is well-known. If $R=\mathbb{Z}/n\mathbb{Z}$, the study of $S_k (\mathbb{Z}/n\mathbb{Z})$ dates back to 1840 \cite{VON} and has been completed in various works \cite{CAR,GMO,MOR2}. More recently, the case $R=\mathbb{Z}/n\mathbb{Z}[i]$ has been solved in \cite{for}. In these cases, we have the following results.

\begin{prop}\label{pre}
Let $k\geq 1$ be an integer.
\begin{itemize}
\item[i)] Finite fields:
$$S_k (\mathbb{F}_q)=\begin{cases}
-1  , & \textrm{if $(q-1) \mid k $ };\\ 0  , &\textrm{otherwise}.
\end{cases}$$
\item[ii)] Integers modulo $n$:
$$S_k (\mathbb{Z}/n\mathbb{Z})=
\begin{cases}  \displaystyle{-\sum_{p \mid n, p-1 \mid k} \frac{n }{p }
    },
  & \textrm{if $k $ is even or $k=1$ or $n \not \equiv 0\pmod{4}$};\\
   \displaystyle{0
    }, & \textrm{otherwise}.
\end{cases}
$$
\item[iii)] Gaussian integers modulo $n$:
$$
S_k (\mathbb{Z}/n\mathbb{Z}[i])=
\begin{cases}  \frac{n}{2}(1+i),
  & \textrm{if $k>1$ is odd and $n\equiv 2\pmod{4}$};\\
  \displaystyle{-\sum_{p\in\mathcal{P}(k,n)} \frac{n^2}{p^2}}, & \textrm{otherwise}.
\end{cases}
$$
where $$
  \mathcal{P}(k,n):=\{ \textrm{prime $p$} : p \mid \mid n, p^2-1\mid k,
  p\equiv 3\pmod{4}\}
  $$
  and $p\mid\mid n$ means that $p\mid n$, but $p^2\nmid n$.
\end{itemize}
\end{prop}

On the other hand, there are not many works dealing with the non-commutative case. The case of square matrices over finite fields is studied in \cite{BCL}, where the following result is proved.

\begin{prop}
Let $k,d\geq 1$ be integers. Then $S_k(\mathbb{M}_d(\mathbb{F}_q))=0$ unless $q=2=d$ and $1<k\equiv -1,0,1 \pmod{6}$ in which case $S_k (\mathbb{M}_d(\mathbb{F}_q))=I_2$.
\end{prop}

In \cite{for2} some general results for $S_k(\mathbb{M}_d(\mathbb{Z}/n\mathbb{Z}))$ are proved and the following conjecture is made about power sums of matrix rings over finite commutative rings.

\begin{con}
Let $d>1$ and let $R$ be a finite commutative ring. Then, $S_k(\mathbb{M}_d(R))=0$ unless the following conditions hold:
\begin{enumerate}
\item[i)] $d=2$,
\item[ii)] $|R| \equiv 2 \pmod{4}$ and $1<k\equiv -1,0,1 \pmod{6}$,
\item[iii)] The unique element $e\in R \setminus \{0\}$ such that $2e =0$ is idempotent.
\end{enumerate}
Moreover, in this case
$$ S_k(\mathbb{M}_d(R))=\begin{pmatrix} e & 0 \\ 0 & e \end{pmatrix}.$$
\end{con}

In this paper, we completely solve the problem of computing $S_k(R)$ for any finite commutative unital ring. To do so, we first recall some well-known ring-theoretic facts. If $R$ is a finite commutative unital ring with $\textrm{char}(R)=n=p_1^{t_1}\cdots p_l^{t_l}$, then we have a decomposition of $R$ as a direct product of rings
\begin{equation}\label{des1}R\cong R_1\times\cdots\times R_l,\end{equation}
where $\textrm{char}(R_i)=p_i^{t_i}$. Moreover, for every $i\in\{1,\dots,l\}$, the ring $R_i$ can be seen as a $\mathbb{Z}/p_i^{t_i}\mathbb{Z}$-module and hence it can be decomposed as a direct sum of cyclic modules
\begin{equation}\label{des2}R_i\cong R_{i,1}\oplus\cdots\oplus R_{i,m_i},\end{equation}
with $\textrm{ann}(R_{i,j})=(p_i^{t_j})$ and $1\leq t_j\leq t_i$.

\section{The prime-power characteristic case}

Decomposition (\ref{des1}) above implies that, in order to get a general result, we can restrict ourselves to the prime-power characteristic case. Hence, throughout this section $R$ will be a finite commutative unital ring with $\textrm{char}(R)=p^t$.

Due to decomposition (\ref{des2}), $R$ is the direct sum of cyclic modules. If $R$ is itself a cyclic $\mathbb{Z}/p^t\mathbb{Z}$-module, then $R\cong\mathbb{Z}/p^t\mathbb{Z}$ and Proposition \ref{pre} ii) applies to obtain that $S_k(R)=-p^{t-1}$ if $p-1 \mid k$ and $S_k(R)=0$ otherwise. Thus, we will focus on the non-cyclic case.

First of all, we will prove that if $t>1$; i.e., if the characteristic is a prime-power but not a prime then $S_k(R)=0$.

\begin{prop}\label{nocic}
Let $R$ be a finite commutative unital ring such that $\textrm{char}(R)=p^t$ with $t>1$. If $R$ is not a cyclic $\mathbb{Z}/p^t\mathbb{Z}$-module, then $S_k(R)=0$ for every $k\geq 1$.
\end{prop}
\begin{proof}
Due to decomposition (\ref{des2}) we have that $R\cong R_1\oplus\cdots\oplus R_m$ with $m\geq 2$, $R_i$ cyclic and $\textrm{ann}(R_i)=(p^{t_i})$ with $1\leq t_i\leq t$.

Hence, if we denote by $x_i$ a generator of $R_i$, then every element of $R$ can be uniquely written in the form $a_1 x_1+\cdots+a_m x_m$ with $a_i \in \{0,\dots,p^{t_i}-1\}$ for every $i\in\{1,\dots,m\}$. Thus,
$$S_k(R)=\sum_{a_i=0}^{p^{t_i}-1}(a_1 x_1+\cdots+a_m x_m)^k=\sum_{s=0}^k \sum_{a_i=0}^{p^{t_i}-1} \binom {k}{s}(a_1 x_1)^s (a_2x_2+\cdots+a_m x_m)^{k-s}$$
and we will proceed inductively.

First of all, assume that $m=2$. In this case,
$$S_k(R)=\sum_{s=0}^k \binom {k}{s}\sum_{a_1=0}^{p^{t_1}-1}\sum_{a_2=0}^{p^{t_2}-1} (a_1 x_1)^s (a_2x_2)^{k-s}$$
and note that either $t_1\geq 2$ or $t_2\geq 2$, for if $t_1,t_2<2$ then $t=1$ which is not possible.

For every $s\in\{0,\dots, k\}$, denote by $\displaystyle A(s):=\sum_{a_1=0}^{p^{t_1}-1}\sum_{a_2=0}^{p^{t_2}-1}(a_1 x_1)^s (a_2x_2)^{k-s}$. Now, since $p^{t_i}x_i=0$ we have that
$$A(0)=\sum_{a_1=0}^{p^{t_1}-1}\sum_{a_2=0}^{p^{t_2}-1}(a_2x_2)^{k}=p^{t_2}\sum_{a_2=0}^{p^{t_2}-1}(a_2x_2)^{k}=0,$$
$$A(k)=\sum_{a_1=0}^{p^{t_1}-1}\sum_{a_2=0}^{p^{t_2}-1}(a_1 x_1)^k=p^{t_1}\sum_{a_1=0}^{p^{t_1}-1}(a_1 x_1)^k=0.$$
On the other hand, if $0<s<k$,
$$A(s)=\sum_{a_1=0}^{p^{t_1}-1}(a_1 x_1)^s \sum_{a_2=0}^{p^{t_2}-1}(a_2x_2)^{k-s}$$
and due to Proposition \ref{pre} ii) we have that $\displaystyle\sum_{a_1=0}^{p^{t_1}-1}(a_1 x_1)^s$ is either $0$ or $-p^{t_1-1}x_1^s$ and also that $\displaystyle\sum_{a_2=0}^{p^{t_2}-1}(a_2x_2)^{k-s}$ is either $0$ or $-p^{t_2-1}x_2^{k-s}$. Consequently, it follows that either $A(s)=0$ or $A(s)=p^{t_1+t_2-2}x_1^sx_2^{k-s}$ but in this latter case, since either $t_1\geq 2$ or $t_2\geq 2$, it also follows that $A(s)=0$ as claimed.

Now, if $m>2$,
$$S_k(R)=\sum_{s=0}^k \binom {k}{s}\sum_{a_1=0}^{p^{t_1}-1} (a_1 x_1)^s \sum_{\substack{a_i=0\\ i\neq 1}}^{p^{t_i}-1}(a_2x_2+\cdots+a_m x_m)^{k-s}$$
and for every $0\leq s\leq k$ at least one of the summatories appearing in the latter expression is 0 by induction hypothesis.
\end{proof}

The following series of technical lemmata will be useful when we consider the case $t=1$ in the sequel. In what follows $\mathbb{F}_q$ will denote the field with $q$ elements.

\begin{lem}\label{ndi1}
Let $R$ be a finite commutative unital $\mathbb{F}_q$-algebra with $q>2$ such that there exists $x\in R-\{0\}$ with $x^2=0$. Then, $S_k (R)=0$ for every $k$.
\end{lem}
\begin{proof}
Given $x\in R-\{0\}$ with $x^2=0$ we can decompose $R$ as the direct sum of vector subspaces $R=<x> \oplus \overline{R}$. Thus,
\begin{align*}S_k(R)&=\sum_{a \in \mathbb{F}_q } \sum_{t\in \overline{R}} (ax + t)^k=\sum_{a \in \mathbb{F}_q }  \sum_{t\in \overline{R}}(t^k+kt^{k-1}ax)=\\
&=q \sum_{t\in \overline{R}}t^k+ \sum_{a \in \mathbb{F}_q }a \sum_{t\in \overline{R}}k t^{k-1}x=0,\end{align*}
because, for $q>2$, it holds that $\displaystyle\sum_{a \in \mathbb{F}_q }a =0$ due to Proposition \ref{pre} i).
\end{proof}

\begin{lem}\label{ndi2}
Let $R$ be a finite commutative unital $\mathbb{F}_q$-algebra such that there exist a free family $\{x,y\}$ with $xy=0$. Then, $S_k (R)=0$ for every $k$.
\end{lem}
\begin{proof}
Given a free family $\{x,y\}$ with $xy=0$ we can decompose $R$ as the direct sum of vector subspaces $R=<x>\oplus <y> \oplus \overline{R}$. Thus,
\begin{align*}S_k(R)&=\sum_{a \in \mathbb{F}_q }\sum_{b\in \mathbb{F}_q}  \sum_{t\in \overline{R}} (ax+by+ t)^k=\\
&=\sum_{a \in \mathbb{F}_q }\sum_{b\in \mathbb{F}_q}  \sum_{t\in \overline{R}} \left( t^k + \sum_{s=1}^k \binom{k}{s}(a^sx^s+b^sy^s)t^{k-s}\right)=\\
&=q^2\sum_{t\in\overline{R}} t^k+ q\sum_{a \in \mathbb{F}_q } \sum_{t\in \overline{R}} \sum_{s=1}^k \binom{k}{s} a^sx^st^{k-s}+q\sum_{b \in \mathbb{F}_q } \sum_{t\in \overline{R}} \sum_{s=1}^k \binom{k}{s} b^sy^st^{k-s}=0\end{align*}
as claimed.
\end{proof}

\begin{lem}\label{pol2}
Let $R\cong(\mathbb{Z}/2\mathbb{Z})[x]/(x^2)$ and let $u$ be the only non-zero idempotent of $R$. Then, $S_k(R)=u$ if $k>1$ is odd and $S_k(R)=0$ otherwise.
\end{lem}
\begin{proof}
In this situation, $R=\{0,1,u,1+u\}$ and since $\textrm{char}(R)=2$, we have that
$$S_k(R)=0^k+1^k+u^k+(1+u)^k=ku$$
and the result follows.
\end{proof}

Now, we are in the conditions to prove the main result of this section.

\begin{teor}\label{teorpp}
Let $R$ be a finite commutative unital ring of prime-power characteristic. Then, $S_k(R)\neq 0$ if and only if one of the following conditions hold.
\begin{itemize}
\item[i)] $|R|-1 \mid k$ and $R$ is a field.
\item[ii)] $R \cong \mathbb{Z}/p^s\mathbb{Z}$ and $p-1 \mid k$.
\item[iii)] $R\cong(\mathbb{Z}/2\mathbb{Z})[x]/(x^2)$ and $k>1$ is odd.
\end{itemize}
\end{teor}
\begin{proof}
If condition i) holds, then $S_k(R)\neq 0$ due to Proposition \ref{pre} i). If condition ii) holds, then $S_k(R)\neq 0$ due to Proposition \ref{pre} ii). If condition iii) holds, then $S_k(R)\neq 0$ due to Lemma \ref{pol2}.

Conversely, assume that $S_k(R)\neq 0$. Proposition \ref{nocic} implies that $R$ must be cyclic or it must have prime characteristic. If it is cyclic, then Proposition \ref{pre} i) or ii) applies (depending on wether $R$ is a field or not) and condition i) or ii) holds, respectively. If $R$ is not cyclic and has prime characteristic, Lemma \ref{ndi2} implies that $R$ cannot contain two different zero-divisors. Consequently, not being a field, $R$ must contain idempotents but in this case Lemma \ref{ndi1} implies that $\textrm{char}(R)=2$. But all the previous restrictions lead to $R\cong(\mathbb{Z}/2\mathbb{Z})[x]/(x^2)$ so Lemma \ref{pol2} applies and the result follows.
\end{proof}

As a consequence, we obtain the following corollary that gives a characterization of finite fields in terms of power sums. The proof is straightforward and we omit.

\begin{cor}\label{cor1}
Let $R$ be a finite commutative unital ring such that $|R|$ is a prime-power. Then, the following are equivalent.
\begin{itemize}
\item[i)] $R$ is a field.
\item[ii)] $S_k(R)=0$ or $S_k(R)=-1$  for every $k\geq 1$.
\item[iii)] $S_k(R)=-1$ for $k=|R|-1$.
\item[iv)] There exists $k\geq 1$ such that $S_k(R)=-1$.
\end{itemize}
\end{cor}

\begin{rem}
Corollary \ref{cor1} provides an elementary (although inefficient) algorithm to determine if a polynomial with integer coefficients is irreducible over $\mathbb{Z}/p\mathbb{Z}$. Its computational complexity is exponential, thus it is worse than the fast already known algorithms for the factorization of polynomials \cite{RED}.  Moreover, it also determines if an ideal in $\mathbb{Z}/p\mathbb{Z}[x_1,\dots,x_m]$ is maximal.

For instance, let us consider $I=(1 + x^2 + y^2,-1 - x + y^2)\triangleleft\mathbb{Z}/3\mathbb{Z}[x,y]$. We will prove that this ideal is maximal. To do so, we compute the sum
$$\sum_{0<=a,b,c,d<3} (a+b x +c y + d x y)^{80},$$ doing repeatedly the substitutions $x^n\rightarrow x^{n-2}(2+2y^2)$ and $y^n\rightarrow y^{n-2}(1+x)$ when necessary until we arrive to an expression of the form $A+ B x + C y + D xy$. The ideal $I$ is maximal if and only if $A \equiv 2 \pmod 3$ and $B \equiv C \equiv D \equiv 0 \pmod 3$, which is the case.
\end{rem}

\section{The general case}

In the previous section we have focused on the case when the characteristic of the ring is a prime-power. Now, we will focus on the general case. Let $R$ be a finite commutative unital ring and assume that $|R|=p_1^{s_1}\cdots p_l^{s_l}$. This implies that $\textrm{char}(R)=p_1^{t_1}\cdots p_l^{t_l}$ with $1\leq t_i\leq s_i$ for every $i$. Note that $t_i=s_i$ if and only if $R_i$ isomorphic to $\mathbb{Z}/p_i^{s_i}$. In this situation, we define rings $R_i=R/p_i^{t_i}R$ for every $i\in\{1,\dots,l\}$. Note that $\textrm{char}(R_i)=p_i^{t_i}$ and, moreover, $R\cong R_1\times\cdots\times R_l$ is precisely the decomposition given in (\ref{des1}).

In this setting, given $k\geq 1$, let us define the sets
\begin{align*}
\mathcal{P}_k(R)&:=\{p_i: R_i \textrm{ is a field and } p_i ^{s_i}-1 \mid k \},\\
\overline{\mathcal{P}}_k(R)&:=\{p_i: R_i \textrm{ is isomorphic to }\mathbb{Z}/p_i^{s_i}\mathbb{Z}\textrm{ with $s_i>1$ and }p_i  -1 \mid k  \}.
\end{align*}

The following lemma is straightforward.

\begin{lem}\label{chop}
Let $R_1$ and $R_2$ be finite commutative unital ring and let $R=R_1\times R_2$ be its direct sum. Then,
$$S_k(R)=(|R_2|S_k(R_1),|R_1|S_k(R_2)).$$
\end{lem}

Now, we are in the condition to prove the main result of the paper.

\begin{teor}
Let $R$ be a finite commutative unital ring with $|R|=p_1^{s_1}\cdots p_l^{s_l}$ and let $k\geq 1$ be an integer. Then, with the previous notation
\begin{itemize}
\item[i)] If $k$ is even, then
$$S_k(R)=-\left ( \sum_{p_i \in \mathcal{P}_k} \frac{|R|}{p_i^{s_i}}+ \sum_{p_i \in \overline{\mathcal{P}}_k} \frac{|R|}{p_i} \right).$$
\item[ii)] If $k>1$ is odd and $2 \in \mathcal{P}_k$, then
$$S_k(R)=-\frac{|R|}{2^{\nu_2(|R|)}}.$$
\item[iii)] If $k>1$ is odd and $2 \in \overline{\mathcal{P}}_k$, then
$$S_k(R)=-\frac{|R|}{2}.$$
\item[iv)] If $k>1$ is odd and $R_i \cong (\mathbb{Z}/2\mathbb{Z})[x]/(x^2)$ for some $i$, then
$$S_k(R)=u,$$
where $u$ is the only non-zero nilpotent element of $R$ such that $2u=0$.
\item[v)] If $k=1$ and $R_i \cong \mathbb{Z}/2\mathbb{Z}$ for some $i$, then
$$S_k(R)=-\frac{|R|}{2}.$$
\item[vi)] In any other case, $S_k(R)=0$.
\end{itemize}
\end{teor}
\begin{proof}
First of all, observe that Lemma \ref{chop} implies that
$$S_k(R)=\left(\frac{|R|}{p_1^{s_1}}S_k(R_1),\dots,\frac{|R|}{p_l^{s_l}}S_k(R_l)\right).$$
Now,
\begin{itemize}
\item[i)] If $k$ is even, Theorem \ref{teorpp} implies that $S_k(R_i)=0$ unless $R_i$ is a field with $|R_i|-1\mid k$ or $R_i\cong\mathbb{Z}/p_i^{s_i}\mathbb{Z}$ with $p_i-1\mid k$ (recall that in this case $\textrm{char}(R_i)=|R_i|$). Due to Proposition \ref{pre} i) and ii), in the first case $S_k(R_i)=-1$ while in the second case $S_k(R_i)=-p_i^{s_i-1}$. Hence, the result follows.
\item[ii)] If $k>1$ is odd and $2 \in \mathcal{P}_k$, we can assume without loss of generality that $p_1=2$. Then, Theorem \ref{teorpp} implies that $S_k(R_i)=0$ for every $i\geq 2$ and the result follows from Proposition \ref{pre} i).
\item[iii)] It is enough to reason like in ii) but the result follows from Proposition \ref{pre} ii).
\item[iv)] Again, the same idea as in ii) and iii) but the claim follows from Lemma \ref{pol2}.
\item[v)] The same as in ii), iii) and iv). Note that in this case $2\in\mathcal{P}_k$ and we can apply either Proposition \ref{pre} i) or ii).
\item[vi)] Theorem \ref{teorpp} states that the only cases in which $S_k(R_i)\neq0$ for some $i$ are precisely the previous ones.
\end{itemize}
\end{proof}

Given a finite commutative unital ring $R$, let $\mathfrak{i}:\mathbb{Z}\longrightarrow R$ be the unique ring homomorphism defined by $\mathfrak{i}(1)=1$. The previous result clearly implies that the power sum $S_k(R)$ is an element of $\textrm{Im}(\mathfrak{i})$ unless $R\cong \mathbb{Z}/{2}\mathbb{Z}[x]/(x^2) \times S$ with $k>1$ odd  and $|S|$ odd.

\begin{cor}
If $k>1$ is odd and $R$ contains a unique non-zero nilpotent element $u$ such that $2u=0$, then $S_k(R)=u$. Otherwise, $S_k(R)\in\mathfrak{i}(\mathbb{Z})$.
\end{cor}

Now, we characterize those finite commutative unital rings such that the power sum $S_k(R)$ is a unit.

\begin{cor} \label{unit}
Let $R$ be a finite commutative unital ring and let $k\geq 1$ be an integer. Then, $S_k(R)$ is a unit if and only if the following conditions hold:
\begin{itemize}
\item[i)] There exist fields $F_1,\dots, F_l$ such that $R\cong F_1 \times \cdots \times F_l$.
\item[ii)] $\textrm{char}(F_i)\neq \textrm{char}(F_j)$ for every $i\neq j$.
\item[iii)] $(|F_i|-1)\mid k$ for every $1\leq i\leq l$.
\end{itemize}
\end{cor}
\begin{proof}
Let $R$ be a finite commutative unital ring. We know that $R\cong R_1\times\cdots\times R_l$ with $R_i$ rings with coprime prime-power characteristic. Due to Lemma \ref{chop} it is clear that $S_k(R)$ is a unit in $R$ if and only if $S_k(R_i)$ is a unit in $R_i$ for every $i$. But by Theorem \ref{teorpp} and Proposition \ref{pre}, this happens if and only if conditions i), ii) and iii) hold.
\end{proof}

From the previous results the question of determining those finite commutative unital rings $R$ such that $S_{|R|}(R)=1$ naturally arises. This question generalizes the problem of determining the integers $n$ for which $\displaystyle \sum_{i=1}^n i^n \equiv 1 \pmod{n}$. This latter problem was solved in \cite[Proposition 1]{GOS} where it was proved that  $1, 2, 6, 42$ and $1806$ are the only possibilities.

\begin{teor} \label{uno}
Let $R$ be a finite commutative unital ring. Then, $S_{|R|}(R)=1$ if and only if the following conditions hold:
\begin{itemize}
\item[i)] $R\cong \mathbb{F}_{p_1^{s_1}} \times\cdots\times \mathbb{F}_{p_l^{s_l}}$ with $p_i \neq p_j$ for every $i\neq j$.
\item[ii)] $p_i^{s_i}-1\mid |R|$ for every $1\leq i\leq s$.
\item[iii)] $|R| \equiv -p^{s_i} \pmod{p^{s_i+1}}$ for every $1\leq i\leq l$.
\end{itemize}
\end{teor}
\begin{proof}
If $S_{|R|}(R)=1$, in particular $S_k(R)$ is a unit so Corollary \ref{unit} applies to give conditions i) and ii). Moreover, using Lemma \ref{chop} $S_k(R)=1$ if and only if $\displaystyle \frac{|R|}{p_i^{s_i}}S_k(\mathbb{F}_{p_i^{s_i}})=1$ in $\mathbb{F}_{p_i^{s_i}}$. Since $S_k(\mathbb{F}_{p_i^{s_i}})=-1$ due to condition ii) and Proposition \ref{pre}, condition iii) easily follows.

Conversely, if conditions i), ii) and iii) hold, it is enough to apply Theorem \ref{teorpp} and Proposition \ref{pre} as usual to get the result.
\end{proof}

The following easy corollary relates the previous theorem with \cite[Proposition 1]{GOS}.

\begin{cor}\label{ord}
Let $R$ be a finite commutative unital ring such that $|R|$ is square-free and $S_{|R|}(R)=1$. Then, $R$ is isomoprhic to one of the following rings: the Zero Ring, $\mathbb{Z}/2\mathbb{Z}$, $\mathbb{Z}/6\mathbb{Z}$, $\mathbb{Z}/42\mathbb{Z}$ or $\mathbb{Z}/1806\mathbb{Z}$.
\end{cor}
\begin{proof}
Theorem \ref{uno} above implies that $R\cong \mathbb{Z}/n\mathbb{Z}$ for some square-free integer $n$. Then, $S_k(R)=\displaystyle \sum_{i=1}^n i^n$ and it is enough to apply \cite[Proposition 1]{GOS} to get the result.
\end{proof}

In addition to the aforementioned rings, there is only one more finite commutative unital ring $R$ with order smaller than $10^7$ satisfying $S_{|R|}(R)=1$. Namely,
$$\mathbb{F}_{16}\times \mathbb{F}_{9} \times \mathbb{F}_5.$$

There are not many rings $R$ with $S_{|R|}(R)=-1$ either. The following result, whose proof is identical to that of Theorem \ref{uno} (and hence we omit) characterizes them.

\begin{teor}
Let $R$ be a finite commutative unital ring. Then, $S_{|R|}(R)=-1$ if and only if the following conditions hold:
\begin{itemize}
\item[i)] $R\cong \mathbb{F}_{p_1^{s_1}} \times\cdots\times \mathbb{F}_{p_l^{s_l}}$ with $p_i \neq p_j$ for every $i\neq j$.
\item[ii)] $p_i^{s_i}-1\mid |R|$ for every $1\leq i\leq l$.
\item[iii)] $|R| \equiv p^{s_i} \pmod{p^{s_i+1}}$ for every $1\leq i\leq l$.
\end{itemize}
\end{teor}

We have only been able to find 5 rings with this property: the Zero Ring, $\mathbb{F}_2$, $\mathbb{F}_4 \times \mathbb{F}_3$, $\mathbb{F}_{16} \times\mathbb{F}_{81} \times \mathbb{F}_{25}$, $\mathbb{F}_{16} \times\mathbb{F}_{81}\times\mathbb{F}_{5}\times\mathbb{F}_{11}$. The orders of the non-zero cases are: $2$, $12$, $32400$ and $71280$.

\section{Power sums over $\mathbb{Z}/n\mathbb{Z}[x]/(f(x))$}

As an application of the previous results, we are interested in computing the power sum $S_k(\mathbb{Z}/n\mathbb{Z}[x]/(f(x)))$, where $f(x)$ is a monic polynomial. When $\deg f=1$, the result is straightforward because $\mathbb{Z}/n\mathbb{Z}[x]/(f(x))\cong \mathbb{Z}/n\mathbb{Z}$ and Proposition \ref{pre} ii) applies.

In order to study the case when $\deg f>1$ we will first focus on the quadratic case.

\subsection{Power sums over $\mathbb{Z}/n\mathbb{Z}[x]/(x^2+bx+c)$}
\

Before we proceed, let us introduce some notation. Given any positive integer $n$ and integers $b,c$ we define
$$R_n^{b,c}:=\mathbb{Z}/n\mathbb{Z}[x]/(x^2+bx+c).$$
As usual, to compute the value of $S_k(R_n^{b,c})$ we will first focus on the case when $n$ is a prime power.

\begin{prop}\label{cuadpp}
Let $k\geq 1$ be and integer.
\begin{itemize}
\item[i)]
If $s$ is a positive integer,\\
$\displaystyle
S_k(R_{2^s}^{b,c}) =
\begin{cases}   1    ,
  & \textrm{if $s=1$, $b$ and $c$ are odd and $3 \mid k $;}\\
1+x   ,
  & \textrm{if $s=1$, $b$ is even, $c$ is odd and $k>1$ is odd;}\\
      x   ,
  & \textrm{if $s=1$, $b$ and $c$ are even and $k>1$ is odd; }\\
  0
      & \textrm{otherwise}.
\end{cases}
$
\item[ii)]
If $p$ is an odd prime and $s$ is a positive integer,\\
$\displaystyle
S_k(R_{p^s}^{b,c})=
\begin{cases}   -1   ,   & \textrm{if $s=1$, $p^2-1 \mid k$  and $b^2-4c$ is not a square mod. $p$;} \\
  0
      & \textrm{otherwise}.
\end{cases}
$
\end{itemize}
\end{prop}
\begin{proof}
\begin{itemize}
\item[i)] First of all, if $s>1$ then $\textrm{char}(R_{2^s}^{b,c})\geq 4$ and we can apply Proposition \ref{nocic} to obtain that $S_k(R_{2^s}^{b,c})=0$ for every $k$ in this case.

If $s=1$ and both $b$ and $c$ are even, then $R_{2^s}^{b,c}=\mathbb{Z}/2\mathbb{Z}/(x^2)$ and by Lemma \ref{pol2} it follows that $S_k(R_{2^s}^{b,c})=x$ if $k>1$ is odd and $S_k(R_{2^s}^{b,c})=0$ otherwise.

If $s=1$, $b$ is even and $c$ is odd, then $R_{2^s}^{b,c}=\mathbb{Z}/2\mathbb{Z}/(x^2+1)$ and by Lemma \ref{pol2} it follows that $S_k(R_{2^s}^{b,c})=1+x$ if $k>1$ is odd (note that $1+x$ is the only non-zero nilpotent element) and $S_k(R_{2^s}^{b,c})=0$ otherwise.

If $s=1$, $b$ is odd and $c$ is even, then $R_{2^s}^{b,c}=\mathbb{Z}/2\mathbb{Z}/(x^2+x)$. Since $0=x^2+x=x(x+1)$, we can apply Lemma \ref{ndi2} to obtain that $S_k(R_{2^s}^{b,c})=0$ for every $k$ in this case.

Finally, if both $b$ and $c$ are odd, then $R_{2^s}^{b,c}=\mathbb{Z}/2\mathbb{Z}/(x^2+x+1)\cong\mathbb{F}_4$ because $x^2+x+1$ is irreducible. Hence, we apply Proposition \ref{pre} i) to obtain that $S_k(R_{2^s}^{b,c})=-1=1$ if $3\mid k$ and $S_k(R_{2^s}^{b,c})=0$ otherwise.
\item[ii)] First of all, if $s>1$ then $\textrm{char}(R_{p^s}^{b,c})\geq p^2$ and we can apply Proposition \ref{nocic} to obtain that $S_k(R_{p^s}^{b,c})=0$ for every $k$ in this case.

If $s=1$, observe that $x^2+bx+c$ is reducible if and only if $b^2-4c$ is a quadratic residue modulo $p$. Now, if $x^2+bx+c$ is reducible we can apply Lemma \ref{ndi1} or Lemma \ref{ndi2} to obtain that $S_k(R_{p^s}^{b,c})=0$ for every $k$. Finally, if $x^2+bx+c$ is irreducible then $R_{p^s}^{b,c}\cong\mathbb{F}_{p^2}$ and Proposition \ref{pre} i) ends the proof.
\end{itemize}
\end{proof}

With this proposition, we can prove the general result.

\begin{teor}\label{tquad}
Let $n$ be any positive integer. Given integers $k\geq 1$, $b$ and $c$ we define the following set:
$$\mathcal{P}^{b,c}(k,n):=\{\textrm{prime}\ p: p \mid \mid n,\ p^2-1 \mid k,\ b^2-4c\ \textrm{is not a quadratic residue modulo}\ p\}.$$
Then:
$$
S_k(R_n^{b,c})=
\begin{cases}
  \frac{n}{2},
  & \textrm{if   $b$ and $c$ are odd, $3 \mid k $ and $2\mid\mid n$};  \\
 \frac{n}{2}(1+x),
  & \textrm{if $b$ is even, $c$ is odd, $k>1$ is odd, and $2\mid\mid n$};\\
   \frac{n}{2}  x,
  & \textrm{if $b$ and $c$ are even, $k>1$ is odd and $2\mid\mid n$};\\
  \displaystyle{-\sum_{p\in\mathcal{P}^{b,c}(k,n )} \frac{n^2}{p^2}
    }, & \textrm{otherwise}.
\end{cases}
$$
\end{teor}
\begin{proof}
Observe that for coprime $n_1$ and $n_2$ we have that $R_{n_1n_2}^{b,c}\cong R_{n_1}^{b,c}\times R_{n_2}^{b,c}$. Thus, it suffices to apply Proposition \ref{cuadpp} above.
\end{proof}

As an interesting consequence of this result, we can compute the power sum over the rings $\mathbb{Z}/n\mathbb{Z}[\sqrt{D}]$ for a square-free integer $D$.

\begin{cor}
Let $k,n\geq 1$ be integers and let $D$ be an square-free integer. Consider the set
$$\mathcal{P}(k,n):=\{ \textrm{prime $p$} : p \mid \mid n,\ p^2-1\mid k,\ D\ \textrm{is not a quadratic residue modulo}\ p\}.$$
Then,
$$
S_k(\mathbb{Z}_n[\sqrt{D}]) = \begin{cases}  \frac{n}{2}(1+\sqrt{D}) ,
  & \textrm{if $k>1$ is odd and $2\mid\mid n$};\\
  \displaystyle{-\sum_{p\in\mathcal{P}(k,n)} \frac{n^2}{p^2}
    }, & \textrm{otherwise}.
\end{cases}
$$
\end{cor}
\begin{proof}
Just take $f(x)=x^2-D$ and apply Theorem \ref{tquad}.
\end{proof}

\begin{rem}
If we consider the case $D=-1$, the previous corollary immediately gives Proposition \ref{pre} iii), which was proved in \cite{for} using different, more direct, techniques.
\end{rem}

\subsection{Power sums over $\mathbb{Z}/n\mathbb{Z}[x]/(f(x))$ with $\deg f>2$}
\

Now, we focus on the case when the degree of the considered polynomial is greater than 2. The involved ideas are quite similar to those previously used. We introduce the following notation:
$$R_n^f:=\mathbb{Z}/n\mathbb{Z}[x]/(f(x)).$$

\begin{teor}
Let $f(x)$ be monic polynomial with integer coefficients such that $\deg f>2$ and let $k,n\geq 1$ be integers. Consider the set
$$\mathcal{P}^{f}(k,n):=\{\textrm{prime  $p$} : p \mid \mid n,\ p^{\deg f}-1 \mid k,\ f(x)\ \textrm{is irreducible modulo}\ n\}.   $$
Then,
$$S_k(R_n^f)\equiv  -\sum_{p\in\mathcal{P}^{f}(k,n)} \frac{n^{\deg f}}{p^{\deg f)}}.$$
\end{teor}
\begin{proof}
Let $n=p_1^{s_1} ... p_l^{s_l}$. Then, as usual
$$R_n^f\cong R_{p_1^{s_1}}^f\times\cdots\times R_{p_l^{s_l}}^f$$
an we can apply Lemma \ref{chop}.

First of all, note that $S_k(R_{p_i^{s_i}}^f)=-1$  if $p_i \in \mathcal{P}(n,k)$ and $S_k(R_{p_i^{s_i}}^f)=0$ otherwise. Since $|R_n^f|=n^{\deg f}$, the result follows.
\end{proof}

\section{Future perspectives}

\subsection{Power sums over   non-commutative rings}
A natural sequel for this work would be to focus on the computation of $S_k(R)$ for more general rings. In particular, the non-commutative case seems interesting and we can pose the following question.

\begin{q}\label{Q1}
Is there any finite non-commutative ring $R$ with odd characteristic such that $S_k(R)\neq 0$ for some $k$?
\end{q}

This question is nontrivial and it is enough to restric it to prime characteristic rings. Moreover, we have a lower bound for the cardinality of a candidate to answer Question \ref{Q1} in the affirmative.

\begin{prop}
Let $R$ be a finite non-commutative unital ring with odd characteristic. If $S_k(R)\neq0$ for some $k\geq 1$, then $|R|\geq 81$.
\end{prop}
\begin{proof}
A finite non-commutative unital ring with prime-power characteristic must have cardinality $p^s$ for some prime $p$ and $s>2$. Thus, $|R|=p^s$ con $s>2$. Now, if $|R|=p^3$ we have that
$$R\cong \mathbb{Z}/p\mathbb{Z}<1,x,y>/(x^2=0,y^2=0,xy=x,yx=0)$$
because, up to isomorphism, there exists just one finite non-commutative unital ring with $p^3$ elements. Since, in this case it is easy to see that $S_k(R)=0$ the result follows.
\end{proof}

\subsection{Finite commutative unital rings such that $S_{|R|}(R)=\pm 1$}

At the end of Section 3, we have given the characterizations of finite commutative unital rings satisfying $S_{|R|}(R)=\pm 1$. Those characterizations allowed us to find, by computational means, rings with these properties. Hence, it arises the question of finding some strategies to search for these rings and even to find out if there is a finite number of them (as in the case of square-free $|R|$)

\subsection{Rings such that $S_{|R|-1}(R)=-1$. Generalized Giuga's conjecture}

If $R$ is a field $S_{|R|-1}(R)=-1$. The converse is true (see Corollary \ref{cor1}) if we restric to rings with prime-power characteristic. This immediately suggests the question about the existence of a ring $R$ which is not a field and satisfying that $S_{|R|-1} (R)=-1$. The following result gives a characterization of such a ring.

\begin{teor}
Let $R$ be a finite commutative unital ring. Then, $S_{|R|-1}(R)=-1$ if and only if the following conditions hold:
\begin{itemize}
\item[i)] $R\cong \mathbb{F}_{p_1^{s_1}} \times\cdots\times \mathbb{F}_{p_l^{s_l}}$ with $p_i \neq p_j$ for every $i\neq j$.
\item[ii)] $p_i^{s_i}-1\mid |R|-1$ for every $1\leq i\leq l$.
\item[iii)] $|R| \equiv p^{s_i} \pmod{p^{s_i+1}}$ for every $1\leq i\leq l$.
\end{itemize}
\end{teor}

Condition ii) is satisfied by many integers (by Carmichael numbers, for instance). Nevertheless, we have not been able to find among them any integer satisfying also condition iii). This question is closely related to Giuga's conjecture  \cite{GIU} that states that there are no square-free compound integers satisfying conditions ii) and iii) above. In other words, Giuga's conjecture states that $\displaystyle \sum_{i=1}^n i^{n-1} \equiv -1 \pmod{n}$ if and only if $n$ is prime.
Hence, we propose the following generalization of Giuga's conjecture.

\begin{con} 
Let $R$ be a finite commutative unital ring. Then $\displaystyle\sum_{r \in R} r^{|R|-1}=-1$ if and only if $R$ is a field.
\end{con}

\end{document}